\documentclass[12pt,reqno]{amsart}

\usepackage{multirow}
\usepackage{hhline}

\usepackage{mathtools}
\usepackage{times}
\usepackage[T1]{fontenc}
\usepackage{mathrsfs}
\usepackage{latexsym}
\usepackage[dvips]{graphics}
\usepackage[titletoc, title]{appendix}
\usepackage{amsmath,amsfonts,amsthm,amssymb,amscd}
\usepackage[dvipsnames]{xcolor}
\usepackage{hyperref}
\usepackage{amsmath}

\usepackage{color}
\usepackage{breakurl}

\usepackage{comment}
\newcommand{\bburl}[1]{\textcolor{blue}{\url{#1}}}
\newcommand{\seqnum}[1]{\href{https://oeis.org/#1}{\rm \underline{#1}}}

\newtheorem{thm}{Theorem}[section]

\newtheorem{cor}[thm]{Corollary}
\newtheorem{claim}[thm]{Claim}
\newtheorem{lem}[thm]{Lemma}
\newtheorem{prop}[thm]{Proposition}

\newtheorem{defi}[thm]{Definition}
\newtheorem{rek}[thm]{Remark}

\numberwithin{equation}{section}

\DeclareFontFamily{U}{mathx}{}
\DeclareFontShape{U}{mathx}{m}{n}{<-> mathx10}{}
\DeclareSymbolFont{mathx}{U}{mathx}{m}{n}
\DeclareMathAccent{\widehat}{0}{mathx}{"70}
\DeclareMathAccent{\widecheck}{0}{mathx}{"71}

\begin{document}

\title{On Schreier-type Sets, Partitions, and Compositions}

\author{Kevin Beanland}
\email{\textcolor{blue}{\href{mailto:beanlandk@wlu.edu}{beanlandk@wlu.edu}}}
\address{Department of Mathematics, Washington and Lee University, Lexington, VA 24450, USA}

\author{H\`ung Vi\d{\^e}t Chu}
\email{\textcolor{blue}{\href{mailto:hungchu2@illinois.edu}{hungchu1@tamu.edu}}}
\address{Department of Mathematics, Texas A\&M University, College Station, IL 77843, USA}

\begin{abstract} 
A nonempty set $A\subset\mathbb{N}$ is $\ell$-strong Schreier if $\min A\geqslant \ell|A|-\ell+1$. We define a set of positive integers to be sparse if either the set has at most two numbers or the differences between consecutive numbers in increasing order are non-decreasing. This note establishes a connection between sparse Schreier-type sets and (restricted) partition numbers. One of our results states that if $\mathcal{G}_{n,\ell}$ consists of partitions of $n$ that contain no parts in $\{2, \ldots, \ell\}$, and
\begin{equation*}
\mathcal{A}_{n,\ell}  \ :=\ \{A\subset \{1, \ldots, n\}\,:\,  n\in A, A\mbox{ is sparse and }\ell\mbox{-strong Schreier}\},
\end{equation*}
then 
$$|\mathcal{A}_{n,\ell}|\ =\ |\mathcal{G}_{n-1,\ell}|, \quad n, \ell\in \mathbb{N}.$$
The special case $\mathcal{G}_{n-1, 1}$ consists of all partitions of $n-1$. Besides partitions, integer compositions are also investigated.
\end{abstract}



\thanks{}

\maketitle

\section{Introduction}
A set $A\subset \mathbb{N}$ is said to be \textit{Schreier} if either $A$ is empty or $\min A\geqslant |A|$. Schreier sets are an important concept with many applications in Banach space theory \cite{AG, Od, S}. These sets have also been studied as purely combinatorial objects in Ramsey theory \cite{FN}. In 2012, Bird \cite{B} had a surprising observation that \begin{equation}\label{e1}\#\{A\subset \{1, \ldots, n\}: A\mbox{ is Schreier  and } n\in A\} \ =\ F_n,\end{equation} where $(F_n)_{n\geqslant 1}$ is the Fibonacci sequence with $F_1 = F_2 = 1$ and $F_{n} = F_{n-1}+F_{n-2}$ for $n\geqslant 3$. Subsequent works have modified the Schreier condition or added extra conditions on the set $A$ to discover many unexpected results. For example, counting Schreier sets that are an arithmetic progression gives the number of edges of the (modified) Tur\'{a}n graphs \cite{BCF, C3}. Meanwhile, the general linear relation $q\min A\geqslant p|A|$ gives 
a higher-order linear recurrence of depth $p+q$, where positive and negative terms alternate \cite{BCF, CMX}. Recently, several authors \cite{CIMSZ} generalized \eqref{e1} to Schreier multisets and the $s$-step Fibonacci sequences and moreover, considered the nonlinear Schreier condition $\sqrt[k]{\min A}\geqslant |A|$ for $k\geqslant 2$.

In this note, we introduce (\textit{strongly}) \textit{sparse} Schreier-type sets and connect them to integer partitions. 
\begin{defi}\normalfont
We call a set $A = \{a_1, \ldots, a_n\}\subset\mathbb{N}$ $(a_1 < \cdots < a_n)$ \textit{sparse} if either $|A|\leqslant 2$ or 
$$|A|\ \geqslant\ 3 \mbox{ and }a_{i}-a_{i-1} \ \geqslant\ a_{i-1} - a_{i-2}, \quad 3\leqslant i\leqslant n.$$ 
Furthermore, $A$ is \textit{strongly sparse} if either $|A|\leqslant 2$ or
$$|A|\ \geqslant\ 3 \mbox{ and }a_{i}-a_{i-1} \ >\ a_{i-1} - a_{i-2}, \quad 3\leqslant i\leqslant n.$$
\end{defi}
With an abuse of notation, we may write $a_1 < a_2 < \cdots < a_k$ to mean the set $\{a_1, a_2, \ldots, a_k\}$ with the additional information about the ordering of elements in the set. For $F = a_1 < \cdots < a_n\subset \mathbb{N}$ with $n\geqslant 2$, let
$$\mathcal{D}(F) \ =\ \{a_2-a_1,\ldots,  a_n-a_{n-1}\},$$
which is a multiset. 
Our main results are motivated by the following proposition.

\begin{prop}\label{mp}
Let $m,n\in \mathbb{N}$. Define
$$\mathcal{F}_{n,m}\ =\ \{F\subset \{1,\ldots, n\}\,:\, F\mbox{ is Schreier and sparse }, n\in F, \min F = m\}.$$
Then for $m\leqslant n-1$, 
\begin{equation*}|\mathcal{F}_{n,m+1}| \ =\ p(n-1, m),\end{equation*}
where $p(u, k)$ is the number of partitions of $u$ into exactly $k$ parts.
\end{prop}

\begin{proof}
Let $P(n-1,m)$ consist of all partitions of $n-1$ into exactly $m$ parts. If $m = n-1$, then $|\mathcal{F}_{n,m+1}| = p(n-1,m) = 1$. Assume $n> m+1$. Define a function $\mathcal{R}$ on $F\in \mathcal{F}_{n,m+1}$ as
$$\mathcal{R}(F) \ =\  \{\underbrace{m+1, \ldots, m+1}_{m+1-|F|}\}\cup F.$$
Note that $\mathcal{R}$ gives multisets.  Since $n>m+1$, each set $F$ in $\mathcal{F}_{n,m+1}$ contains at least two numbers. Therefore, $\mathcal{D}$ can take $\mathcal{R}(F)$ as an input. Define $\mathcal{S}:\mathcal{F}_{n,m+1}\rightarrow P(n-1, m)$ as
$\mathcal{S}(F) = \mathcal{D}(\mathcal{R}(F)) + 1$.
Here elements of the set $\mathcal{D}(\mathcal{R}(F)) + 1$ represent parts of a partition.
We show that $\mathcal{S}$ is well-defined and bijective.

Let $F\in \mathcal{F}_{n,m+1}$. Since $\min F = m+1$ and $\max F = n$, the sum of $m$ numbers in $\mathcal{D}(\mathcal{R}(F))$ is $n-m-1$. Therefore, the sum of $m$ parts given by $\mathcal{S}(F)$ is 
$(n-m-1) + m\ =\ n-1$, thus $\mathcal{S}(F)\in P(n-1, m)$, which shows that $\mathcal{S}$ is well-defined. Injectivity is obvious. 
To see surjectivity, choose an arbitrary partition $\{n_1 \leqslant \cdots \leqslant n_m\}$ of $n-1$ with $n_1 \leqslant \cdots \leqslant n_m$. Let $t$ be the smallest (if any) such that $n_t\geqslant 2$. 
It is not hard to see that 
$\mathcal{S}(F) = \{n_1, \ldots, n_m\}$, 
where 
$$F \ =\ \left\{m+1, m+1+(n_t-1), \ldots, m+1+\sum_{i=t}^m (n_i-1)\right\}\ \in\ \mathcal{F}_{n,m+1}.$$
This completes the proof. 
\end{proof}

\begin{cor}\label{c5}
For $n\geqslant 1$, let 
\begin{equation}\label{e61}\mathcal{F}_{n}\ =\ \{F\subset \{1,\ldots, n\}\,:\, F\mbox{ is Schreier and sparse }, n\in F\}.\end{equation}
It holds that $|\mathcal{F}_n| = p(n-1)$, 
where $p(u)$ is the partition number of $u$. Here $p(0) = 1$ by convention. 
\end{cor}

Corollary \ref{c5} shows a nice relation between Schreier-type sets and partition numbers, inspiring us to look at restricted partitions, where certain rules are applied to the parts in each partition. In particular, we are concerned with the following sequences 
\seqnum{A000041} \seqnum{A002865}, \seqnum{A008483}, \seqnum{A008484}, \seqnum{A185325},
\seqnum{A000009}, \seqnum{A025147}, \seqnum{A025148}, \seqnum{A025149}, \seqnum{A025150}, \seqnum{A025151},
\seqnum{A000070}, \seqnum{A027336}, \seqnum{A036469}, \seqnum{A038348}
in the Online Encyclopedia of Integer Sequences (OEIS) \cite{Sl}.

\begin{defi}\normalfont
For $\ell\in \mathbb{N}_{\geqslant 0}$, a finite set $A\subset\mathbb{N}$ is said to be \textit{$\ell$-strong Schreier} if either $A = \emptyset$ or 
$$\min A\ \geqslant\ \ell|A|-\ell+1.$$
\end{defi}
\begin{rek}\normalfont
This definition is natural in the sense that any finite subset of $\mathbb{N}$ is $0$-strong Schreier. If $\ell\geqslant k$, an $\ell$-strong Schreier set is $k$-strong Schreier, and when $\ell = 1$, the $1$-strong Schreier property is the same as the Schreier property. Any set of size $1$ (singleton) is $\ell$-strong Schreier for all $\ell$.
\end{rek}
For $n\geqslant 1$ and $\ell\geqslant 0$, we define the objects to be studied
\begin{itemize}
\item $\mathcal{A}_{n,\ell}  := \{A\subset \{1, \ldots, n\}:  n\in A, A\mbox{ is sparse and }\ell\mbox{-strong Schreier}\}$,
\item $\mathcal{A}^{s}_{n,\ell}  := \{A\subset \{1, \ldots, n\}:  n\in A, A\mbox{ is strongly sparse and }\ell\mbox{-strong Schreier}\}$,
\item $\mathcal{E}_{n,\ell}$: the set of partitions of $n$ into parts which are at least $\ell$, and
\item $\mathcal{E}^{d}_{n,\ell}$: the set of partitions of $n$ into distinct parts, each of which is at least $\ell$. 
\end{itemize}
By convention, $|\mathcal{E}_{0,\ell}| = |\mathcal{E}^d_{0,\ell}| = 1$. For reference, we include tables for initial values of $|\mathcal{A}_{n,\ell}|, |\mathcal{A}^s_{n,\ell}|, |\mathcal{E}_{n,\ell}|$, and $|\mathcal{E}^d_{n,\ell}|$ in the appendix.

\begin{thm}\label{m1}
The following identities hold
\begin{align}\label{e31}|\mathcal{A}_{n,\ell}|&\ =\ \sum_{i=0}^{n-1} |\mathcal{E}_{i,\ell+1}|, \quad n\geqslant 1, \ell\geqslant 0,\mbox{ and }\\
\label{e32}|\mathcal{A}^{s}_{n,\ell}|& \ =\ \sum_{i=0}^{n-1} |\mathcal{E}^{d}_{i,\ell+1}|,\quad n\geqslant 1, \ell\geqslant 0.
\end{align}
\end{thm}

Thanks to Theorem \ref{m1}, we obtain the next theorem, whose short proof utilizes generating functions. 
For $\ell\geqslant 0$, define 
\begin{itemize}
    \item $\mathcal{G}_{n,\ell}$ to consist of partitions of $n$, which have no parts in $\{2,\ldots, \ell\}$, and
    \item $\mathcal{H}_{n,\ell}$ to consist of partitions of $n$, which has no parts in $\{2,\ldots, \ell\}$ and no even parts greater than $2\ell$.
\end{itemize}
Here when $\ell \leqslant 1$, $\{2,\ldots, \ell\} = \emptyset$.
By convention, $|\mathcal{G}_{0,\ell}| = |\mathcal{H}_{0,\ell}| = 1$. The appendix contains instances of $|\mathcal{G}_{n,\ell}|$ and $|\mathcal{H}_{n,\ell}|$.

\begin{thm}\label{c3}
For $n, \ell\in \mathbb{N}$, we have  
\begin{align}
\label{e40}|\mathcal{A}_{n,\ell}|&\ =\ |\mathcal{G}_{n-1,\ell}|\\
\label{e41}|\mathcal{A}^s_{n,\ell}|&\ =\ |\mathcal{H}_{n-1,\ell}|.
\end{align}
\end{thm}

\begin{rek}\normalfont
Note that $(|\mathcal{G}_{n,1})_{n=0}^\infty$ is the sequence of partition numbers. In this case, \eqref{e40} is the same as Corollary \ref{c5}.
The sequence $(|\mathcal{H}_{n,1}|)_{n=0}^\infty$ counts partitions of $n$ that contain no even parts except $2$. Using generating functions, it is not hard to verify that $(|\mathcal{H}_{n,1}|)_{n=0}^\infty$ is also equal to the number of partitions of $n$ that contain at most one even part. This is the sequence \seqnum{A038348}. The two sequences $(|\mathcal{H}_{n,2}|)_{n=0}^\infty$ and $(|\mathcal{H}_{n,3}|)_{n=0}^\infty$ are new and not available in OEIS. 
\end{rek}

A surprising feature of the identities in Theorems \ref{m1} and \ref{c3} is that while the left side involves the $\ell$-strong Schreier property, the right side cleanly counts partitions excluding certain parts.
A natural approach to prove Theorems \ref{m1} and \ref{c3} is to ``divide and conquer". Specifically for \eqref{e31}, we may divide $\mathcal{A}_{n,\ell}$ into $n$ collections of sets based on their minimum and hope that the cardinality of each collection is equal to a corresponding summand on the right side. This approach works in the simplest case when $\ell = 0$, but the naive division fails in general.



Similarly, we may attempt to prove \eqref{e40} by grouping $\mathcal{A}_{n,\ell}$ into sets having the same minimum, while dividing $\mathcal{G}_{n-1, \ell}$ into groups of partitions having the same number of parts. This idea works when $\ell = 1$ as we have seen in the proof of Proposition \ref{mp}. However, the idea again fails for $\ell\geqslant 2$. 

The above discussion suggests that to establish Theorems \ref{m1} and \ref{c3}, we need a more clever way to divide and conquer, which we shall show in the next section.

We end this note with a result about compositions of positive integers. A composition of $n$ is a way to write $n$ as a sum of positive integers, and two sums whose terms are ordered differently constitute two different compositions. The definition of compositions is thus different from that of partitions, whose order of terms does not matter. We shall connect $\ell$-strong Schreier sets to compositions by removing the sparse condition on sets in the definition of $\mathcal{A}_{n,\ell}$. For $m, n\in  \mathbb{N}$ and $\ell\geqslant 0$, define
\begin{align*}
    \mathcal{B}_{n,\ell} &\ =\ \{B\subset \{1, \ldots, n\}:  n\in B, B\mbox{ is }\ell\mbox{-strong Schreier}\}\\
    \mathcal{B}_{n,\ell,m} &\ =\ \{B\subset \{1, \ldots, n\}:  n\in B, |B| = m, B\mbox{ is }\ell\mbox{-strong Schreier}\}.
\end{align*}
\begin{thm}\label{m2}
For $m, n\in  \mathbb{N}$ and $\ell\geqslant 0$, 
$$|\mathcal{B}_{n, \ell, m}| \ =\ c(n+\ell, \ell+1, m),$$
where $c(u, v, s)$ is the number of compositions of $u$ into exactly $s$ parts, each of which is at least $v$. As a result,
$$|\mathcal{B}_{n,\ell}|\ =\ c(n+\ell, \ell+1),$$
where $c(u,v)$ counts compositions of $u$ into parts that are at least $v$. 
\end{thm}


\section{The $\ell$-strong Schreier sets and partitions}

We set up some notation to state a key lemma for the proof of Theorem \ref{m1}. Choose $n\geqslant 1$ and $\ell\geqslant 0$. 
Let $\mathcal{E}_{n,\ell, k}$ consist of partitions in $\mathcal{E}_{n, \ell}$ with exactly $k$ parts. For $k\geqslant 2$ and $q\geqslant 0$, let $\mathcal{E}_{n,\ell, k,q}$ be the collection of all partitions in $\mathcal{E}_{n,\ell, k}$, where the difference between the largest and the second largest parts is $q$.

Define the collection $\mathcal{F}_{n,\ell}$ to consist of sets of positive integers $A = \{a_1, \ldots, a_p, n+1\}$ such that 
\begin{enumerate}
\item[i)] $p\geqslant 2$ and $a_1 < \cdots < a_p < n+1$, 
\item[ii)] $A$ is sparse and $\ell$-strong Schreier, and 
\item [iii)] $n+1+a_{p-1} = 2a_p$.
\end{enumerate}
Note that $\mathcal{F}_{n,\ell} = \emptyset$ for all $n\leqslant 2\ell+1$. 
The first $n$ for which $\mathcal{F}_{n,\ell}$ is nonempty is $2\ell+2$; in fact, $\mathcal{F}_{2\ell+2, \ell} = \{\{2\ell+1, 2\ell+2, 2\ell+3\}\}$. 

Let $\mathcal{F}_{n,\ell,k}$ be the collection of sets in $\mathcal{F}_{n,\ell}$ with exactly $k$ elements. Finally, for $k, q\geqslant 1$, let $\mathcal{F}_{n,\ell, k,q}$ consist of sets in $\mathcal{F}_{n,\ell, k}$ with the smallest element equal to $q$.

\begin{rek}\normalfont
From the definition of $\mathcal{E}_{n,\ell+1}$, for $\ell+1\leqslant n\leqslant 2\ell+1$, we have $|\mathcal{E}_{n,\ell+1}| = 1$, as the only permissible partition of $n$ is itself. When $n = 2\ell+2$, we have two partitions in $\mathcal{E}_{2\ell+2, \ell+1}$, which are $2\ell+2$ and $(\ell+1) + (\ell+1)$. This, along with the discussion after the definition of $\mathcal{F}_{n,\ell}$, makes us suspect that 
$|\mathcal{F}_{n,\ell}|+1 = |\mathcal{E}_{n,\ell+1}|$ for $n\geqslant \ell+1$. Our first goal in this section is to establish this identity (see Corollary \ref{c2}).  
\end{rek}

\begin{lem}\label{l1}
For $n\geqslant \ell+1$, $k\geqslant 2$, and $q\geqslant 1$, we have
\begin{equation}\label{e39}|\mathcal{F}_{n, \ell, k+1, \ell k+q}| \ =\ |\mathcal{E}_{n, \ell+1, k, q-1}|.\end{equation}
\end{lem}
\begin{proof}
Given a nonempty set $A$ of natural numbers, we let $$\mathcal{R}(A)  \ =\ (A\backslash \{\max A\})\cup \{\max A + q-1\}.$$
We define a bijective map $\mathcal{S}: \mathcal{F}_{n,\ell, k+1, \ell k + q}\rightarrow \mathcal{E}_{n, \ell+1, k, q-1}$ as 
$$\mathcal{S}(F) \ =\ \mathcal{R}(\mathcal{D}(F)+\ell).$$
We show that $\mathcal{S}$ is well-defined and bijective.
\begin{enumerate}
\item[a)] Let $F\in \mathcal{F}_{n,\ell, k+1, \ell k + q}$. The sum of $k$ elements in $\mathcal{D}(F)$ is $n+1-(\ell k+q)$. Hence, the sum of $k$ elements in $\mathcal{S}(F)$ is $n$. By Condition iii) in the definition of $\mathcal{F}_{n,\ell}$, we know that the difference of the two largest elements in $\mathcal{S}(F)$ is $q-1$. Hence, $\mathcal{S}(F)\in \mathcal{E}_{n, \ell+1, k, q-1}$.
\item[b)] To show that $\mathcal{S}$ is injective, it suffices to verify that given $F_1\neq F_2\in \mathcal{F}_{n,\ell, k+1, \ell k + q}$, we have $\mathcal{D}(F_1)\neq \mathcal{D}(F_2)$. For $i\in \{1,2\}$, write $F_i = \{f_{i,1},\ldots, f_{i, k+1}\}$. Let $t\geqslant 2$ be the smallest such that $f_{1,t}\neq f_{2,t}$. Since $F_i$ is sparse, 
$$\mathcal{D}(F_i)\ =\ f_{i,2}-f_{i,1}\ \leqslant\ \cdots\ \leqslant \ f_{i,t}-f_{i,t-1}\ \leqslant\ \cdots, \quad i = 1,2.$$
It follows that $t$ is the smallest such that $f_{1,t} - f_{1,t-1}\neq f_{2,t}-f_{2,t-1}$. The above arrangement of numbers in $\mathcal{D}(F_i)$ in increasing order guarantees that $\mathcal{D}(F_1)\neq \mathcal{D}(F_2)$. 
\item[c)] Let $\ell+1\leqslant n_1\leqslant\cdots\leqslant n_k$ be such that $\sum_{i=1}^k n_i = n$ and $n_k-n_{k-1} = q-1$. Define $m_i = n_i-\ell$ for $i\leqslant k-1$ and $m_k = n_k-\ell-(q-1)$. Consider the set
$$F \ =\ \left\{\ell k + q, \ell k + q + m_1, \ldots, \ell k + q + \sum_{i=1}^k m_i\right\}.$$
Observe that
\begin{align*}
\ell k + q +  \sum_{i=1}^k m_i&\ =\ \ell k + q + \sum_{i=1}^{k-1} m_i + m_k\\
    &\ =\  \ell k + q + \sum_{i=1}^{k-1}n_i - \ell(k-1) + n_k-\ell - (q-1)\\
    &\ =\ n+1.
\end{align*}
Hence, $F\in  \mathcal{F}_{n,\ell, k+1, \ell k + q}$. It is easy to see that $\mathcal{S}(F) = n_1\leqslant \cdots \leqslant n_k$. 
\end{enumerate}
We have completed the proof. 
\end{proof}

\begin{cor}\label{c1}
For $n\geqslant \ell+1$ and $k\geqslant 2$, we have $|\mathcal{E}_{n, \ell+1, k}| = |\mathcal{F}_{n,\ell, k+1}|$.
\end{cor}

\begin{proof}
Simply observe that
$$|\mathcal{E}_{n,\ell+1, k}| \ =\ \sum_{q=1}^\infty |\mathcal{E}_{n,\ell+1, k,q-1}|\mbox{ and }|\mathcal{F}_{n,\ell, k+1}|\ =\ \sum_{q=1}^\infty|\mathcal{F}_{n, \ell, k+1, \ell k+q}|,$$
and use Lemma \ref{l1}.
\end{proof}

\begin{cor}\label{c2}
For $n\geqslant \ell+1$, we have $|\mathcal{E}_{n,\ell+1}| = |\mathcal{F}_{n,\ell}|+1$.
\end{cor}
\begin{proof}
By Corollary \ref{c1}, 
\begin{align*}
|\mathcal{E}_{n,\ell+1}| \ =\ \sum_{k=1}^\infty|\mathcal{E}_{n,\ell+1, k}|&\ =\ |\mathcal{E}_{n,\ell+1, 1}| + \sum_{k=2}^\infty |\mathcal{E}_{n,\ell+1, k}|\ =\ |\mathcal{E}_{n,\ell+1, 1}| + \sum_{k=2}^\infty |\mathcal{F}_{n,\ell, k+1}|\\
&\ =\ |\mathcal{E}_{n,\ell+1, 1}| + \sum_{k=3}^\infty |\mathcal{F}_{n,\ell, k}|\ =\ 1 + |\mathcal{F}_{n,\ell}|,
\end{align*}
because sets in $\mathcal{F}_{n,\ell}$ have at least three elements. Note also that $|\mathcal{E}_{n,\ell+1, 1}| = 1$, because $n\geqslant \ell+1$.
\end{proof}

\begin{proof}[Proof of Theorem \ref{m1}, Identity \eqref{e31}]
Pick $n\geqslant 1$ and $\ell\geqslant 0$. We first consider the case $n\geqslant \ell+1$.
To evaluate $|\mathcal{A}_{n+1,\ell}| - |\mathcal{A}_{n,\ell}|$, we define an injective map $f: \mathcal{A}_{n,\ell}\rightarrow \mathcal{A}_{n+1, \ell}$ as 
$$f(A)\ :=\ (A\backslash \{n\})\cup \{n+1\}.$$
It is easy to see that $f(\mathcal{A}_{n,\ell}) = \{\{n+1\}\}\cup \mathcal{C}_n\cup \mathcal{D}_n$, where 
\begin{align*}
    \mathcal{C}_{n}&\ :=\  \{\{m,n+1\}\,:\, \ell+1\leqslant m\leqslant n-1\}, \mbox{ and }\\
    \mathcal{D}_n&\ :=\ \{\{a_1, \ldots, a_p, n+1\}\,:\, p\geqslant 2, a_1 < \cdots < a_p < n \mbox{ is }\ell\mbox{-strong Schreier, sparse}\}.
\end{align*}
Hence,
$$|\mathcal{A}_{n+1, \ell}|-|\mathcal{A}_{n,\ell}|\ =\ |\mathcal{A}_{n+1,\ell}| - |f(\mathcal{A}_{n,\ell})|\ =\ |\mathcal{A}_{n+1, \ell}\backslash (\{\{n+1\}\}\cup \mathcal{C}_n\cup \mathcal{D}_n)|.$$
From the definition of $\mathcal{A}_{n+1, \ell}$, $\mathcal{C}_n$, and $\mathcal{D}_n$, we know that
$\mathcal{A}_{n+1, \ell}\backslash (\{\{n+1\}\}\cup \mathcal{C}_n\cup \mathcal{D}_n)$ is the collection consisting of $\{n, n+1\}$ and sets of the form $\{a_1, \ldots, a_p, n+1\}$ such that
\begin{enumerate}
\item[i)]  $p\geqslant 2$ and $a_1 < \cdots < a_p < n+1$,
\item[ii)] $\{a_1, \ldots, a_p, n+1\}$ is sparse and $\ell$-strong Schreier, and 
\item[iii)] if $a_p < n$, then $\{a_1, \ldots, a_p, n\}$ is not sparse.
\end{enumerate}
\begin{claim}\label{cl1}
Let $A = \{a_1, \ldots, a_p, n+1\}$ satisfy Conditions i) and ii). Then $A$ satisfies Condition iii) if and only if $n+1+a_{p-1} = 2a_p$.
\end{claim}
\begin{proof}

Suppose that $n + 1+ a_{p-1} = 2a_p$; equivalently, $n-a_p = a_p-a_{p-1}-1$. It follows that if $a_p < n$, then $\{a_1, \ldots, a_p, n\}$ is not sparse.

Conversely, assume that Condition iii) holds. Since $A$ is sparse, we know that
\begin{equation*}(n+1) - a_p \ \geqslant \ a_p-a_{p-1}.\end{equation*}
Suppose, for a contradiction, that $(n+1) - a_p > a_p-a_{p-1}$; equivalently, 
$$n\ \geqslant\ a_p+(a_p-a_{p-1})\ >\ a_p.$$
According to Condition iii), $\{a_1, \ldots, a_p, n\}$ is not sparse, which, along with the fact that $A$ is sparse, implies 
\begin{equation*}n-a_p \ <\ a_p - a_{p-1};\end{equation*}
hence, $a_p-a_{p-1}$ lies strictly between two consecutive integers $n-a_p$ and $n+1-a_p$, a contradiction. Therefore, $n+1  + a_{p-1} \ = \ 2a_p$.
\end{proof}

Claim \ref{cl1} states that we can replace Condition iii) by the condition $n + a_{p-1} = 2a_p$. 
Considering the definition of $\mathcal{F}_{n,\ell}$, we have shown that
$$\mathcal{A}_{n+1, \ell}\backslash (\{\{n+1\}\}\cup \mathcal{C}_n\cup \mathcal{D}_n) \ =\  \mathcal{F}_{n,\ell} \cup \{\{n, n+1\}\}.$$
Therefore, $|\mathcal{A}_{n+1, \ell}| - |\mathcal{A}_{n,\ell}| = |\mathcal{F}_{n,\ell}| + 1$. Corollary \ref{c2} then gives
$$|\mathcal{A}_{n+1, \ell}| - |\mathcal{A}_{n, \ell}|\ =\ |\mathcal{E}_{n,\ell+1}|, \quad n\geqslant \ell+1.$$

We consider $1\leqslant n\leqslant \ell$.
When $n\leqslant \ell$, $|\mathcal{E}_{n,\ell+1}| = 0$. On the other hand, $\mathcal{A}_{n, \ell} = \{\{n\}\}$ for all $n\leqslant \ell+1$. Indeed, if there exists $A\in \mathcal{A}_{n,\ell}$ with $|A|\geqslant 2$, then 
$$\min A\ \geqslant\ \ell|A|-\ell+1 \ \geqslant\ \ell+1,$$
so $\max A\geqslant \min A + 1 = \ell+2$. However, $n\leqslant \ell+1$; hence, no such set $A$ exists. Therefore, similar to $|\mathcal{E}_{n,\ell+1}|$, $|\mathcal{A}_{n+1, \ell}|-|\mathcal{A}_{n,\ell}| = 0$ for all $n\leqslant \ell$. This completes our proof that
$$|\mathcal{A}_{n+1, \ell}| - |\mathcal{A}_{n, \ell}|\ =\ |\mathcal{E}_{n,\ell+1}|, \quad n\geqslant 1,$$
which implies \eqref{e31}.
\end{proof}

The proof of Identity \eqref{e32} is similar to the proof of Identity \eqref{e31} with obvious modifications. Specifically, we change the third condition on sets in the family $\mathcal{F}_{n,\ell}$ to 
$$n+a_{p-1} \ =\ 2a_p.$$
The change leads to the definition of $\mathcal{F}^s_{n,\ell}$ and $\mathcal{F}^s_{n, \ell, k, q}$, which are the counterparts of $\mathcal{F}_{n,\ell}$ and $\mathcal{F}_{n,\ell, k, q}$, respectively. The counterpart of \eqref{e39} is 
$$|\mathcal{F}^s_{n,\ell, k+1, \ell k+q}| \ =\ |\mathcal{E}^d_{n, \ell, k, q}|.$$
Note that the index $q-1$ is changed to $q$.
For conciseness, we leave the details for interested readers.

\begin{proof}[Proof of Theorem \ref{c3}]
Fix $\ell\geqslant 1$. We prove \eqref{e40}. 
The generating function for the sequence $(|\mathcal{E}_{n,\ell+1}|)_{n=0}^\infty$ is 
$\Psi(x) = \prod_{i=\ell+1}^\infty (1-x^i)^{-1}$.
Meanwhile, for $n\geqslant 2$, if we add the part $1$ to a partition in $\mathcal{G}_{n-1, \ell}$, we have a partition in $\mathcal{G}_{n, \ell}$. This map is clearly injective. Hence, for each $n\geqslant 2$, $|\mathcal{G}_{n, \ell}|-|\mathcal{G}_{n-1, \ell}|$ counts the number of partitions of $n$ that contain no parts in $\{1, \ldots, \ell\}$, the generating function for which is again $\Psi(x)$. Therefore, $|\mathcal{G}_{n,\ell}| - |\mathcal{G}_{n-1, \ell}| = |\mathcal{E}_{n,\ell+1}|$ for all $n\geqslant 2$. When $n = 1$, we also have $|\mathcal{G}_{1,\ell}| - |\mathcal{G}_{0, \ell}| = |\mathcal{E}_{1,\ell+1}| = 0$, because $\ell\geqslant 1$. Hence, $|\mathcal{G}_{n,\ell}| - |\mathcal{G}_{n-1, \ell}| = |\mathcal{E}_{n,\ell+1}|$ for  $n\in \mathbb{N}$. By \eqref{e31}, 
$|\mathcal{A}_{n,\ell}| = |\mathcal{G}_{n-1,\ell}|$, as desired.

Next, we prove \eqref{e41}. The generating function for $(|\mathcal{E}^d_{n,\ell+1}|)_{n=0}^\infty$ is 
$$
\Psi(x) \ =\ \prod_{i=\ell+1}^\infty(1+x^{i})\ =\ \prod_{i=\ell+1}^\infty\frac{1-x^{2i}}{1-x^{i}}\ =\ \prod_{i=\ell+1}^{2\ell}\frac{1}{1-x^i}\prod_{j=\ell}^\infty\frac{1}{1-x^{2j+1}}.
$$
We establish the generating function for the sequence $(|\mathcal{H}_{n,\ell}|-|\mathcal{H}_{n-1,\ell}|)_{n\geqslant 1}$.
As above, $|\mathcal{H}_{n,\ell}|-|\mathcal{H}_{n-1,\ell}|$ counts the number of partitions of $n$ that neither contain a part in $\{1,\ldots, \ell\}$ nor contain an even part greater than $2\ell$, the generating function for which is
\begin{align*}
\Theta(x) &\ =\ \prod_{i=1}^{\ell}(1+x^{\ell+i}+x^{2(\ell+i)}+\cdots)\prod_{j=\ell}^{\infty}(1+x^{2j+1}+x^{2(2j+1)}+\cdots)\\
&\ =\ \prod_{i=\ell+1}^{2\ell}\frac{1}{1-x^i}\prod_{j=\ell}^\infty\frac{1}{1-x^{2j+1}}\ =\ \Psi(x).
\end{align*}
Therefore, $|\mathcal{H}_{n,\ell}|-|\mathcal{H}_{n-1,\ell}| = |\mathcal{E}^d_{n,\ell+1}|$ for $n\geqslant 2$. When $n = 1$, 
$|\mathcal{H}_{1,\ell}| - |\mathcal{H}_{0,\ell}|  = |\mathcal{E}^d_{1,\ell+1}| = 0$,
because $\ell\geqslant 1$. We have shown
$|\mathcal{H}_{n,\ell}|-|\mathcal{H}_{n-1,\ell}|  = |\mathcal{E}^d_{n,\ell+1}|$ for $n\in \mathbb{N}$.
Use \eqref{e32} to obtain $|\mathcal{H}_{{n-1},\ell}| = |\mathcal{A}_{n,\ell}^s|$, as desired. 
\end{proof}


\section{The $\ell$-strong Schreier sets and compositions}

Our proof of Theorem \ref{m2} invokes two well-known results: the star-and-bar problem and the hockey-stick identity. We refer the readers to \cite[Lemma 2.1]{KKMW} for the former and to \cite[Theorem 1.2.3 item (5)]{W} for the latter. 

\begin{proof}[Proof of Theorem \ref{m2}]
Observe that $\mathcal{B}_{n,\ell, 1} = \{n\}$ and $c(n+\ell, \ell+1, 1) = 1$, so $|\mathcal{B}_{n,\ell,1}| = c(n+\ell, \ell+1, 1)$. Assume that $m\geqslant 2$. 
According to the star-and-bar problem, 
$$c(n+\ell, \ell+1, m)\ =\ \binom{n+\ell-(\ell+1)m+(m-1)}{m-1}\ =\ \binom{n+\ell-\ell m-1}{m-1}.$$
On the other hand, for $B\in \mathcal{B}_{n,\ell,m}$, the $\ell$-strong Schreier property implies that $\min B\geqslant \ell m-\ell  + 1$. A set in $\mathcal{B}_{n,\ell,m}$ is formed by first choosing its minimum $i\geqslant \ell m-\ell+1$ then choosing $m-2$ number(s) strictly between $i$ and $n$; hence, 
$$|\mathcal{B}_{n,\ell, m}|\ =\ \sum_{i=\ell m-\ell+1}^{n-1}\binom{n-i-1}{m-2}\ =\ \sum_{i = 0}^{n+\ell-\ell m-2}\binom{i}{m-2}.$$
We, therefore, have $|\mathcal{B}_{n,\ell, m}|\neq 0$ if and only if $n-i-1\geqslant m-2$ for some $i\geqslant \ell m-\ell+1$. In other words, $|\mathcal{B}_{n,\ell, m}|\neq 0$ if and only if $n+\ell\geqslant (\ell+1)m$. Similarly, $c(n+\ell, \ell+1, m)\neq 0$ if and only if $n+\ell\geqslant (\ell+1)m$. Therefore, it suffices to prove that $|\mathcal{B}_{n,\ell, m}| = c(n+\ell, \ell+1, m)$ when $n+\ell\geqslant (\ell+1)m$. In this case, applying the hockey-stick identity, we obtain 
$$|\mathcal{B}_{n,\ell, m}|\ =\ \sum_{i = 0}^{n+\ell-\ell m-2}\binom{i}{m-2}\ =\ \binom{n+\ell-\ell m - 1}{m-1}\ =\ c(n+\ell, \ell+1, m).$$
This completes our proof. 
\end{proof}

\section{Appendix}
Below are tables for initial values of $|\mathcal{A}_{n,\ell}|, |\mathcal{A}^s_{n,\ell}|, |\mathcal{E}_{n,\ell}|$, $|\mathcal{E}^d_{n,\ell}|$, $|\mathcal{G}_{n,\ell}|$, and $|\mathcal{H}_{n,\ell}|$.

\begin{tabular}{c|cccccccccccccccc}
$n$ & $1$ & $2$ & $3$ & $4$ & $5$ & $6$ & $7$ & $8$ & $9$ & $10$ & $11$ & $12$ & $13$ & $14$ & $15$ & $16$     \\
\hline
$|\mathcal{A}_{n,0}|$ & 1 & 2 & 4 & 7        & $12$ &    $19$ &   $30$ & $45$ & $67$ & $97$ & $139$ & $195$ & $272$ & $373$ & $508$ & $684$ \\
$|\mathcal{A}_{n,1}|$ & 1 & 1 & $2$ & $3$    & $5$ &   $7$&   $11$ & $15$ & $22$ & $30$ & $42$ & $56$ & $77$ & $101$ & $135$ & $176$ \\
$|\mathcal{A}_{n,2}|$ & 1 & 1 & $1$ & $2$   & $3$ &   $4$&   $6$ & $8$ & $11$ & $15$ & $20$ & $26$ & $35$ & $45$ & $58$ & $75$ \\
$|\mathcal{A}_{n,3}|$ & 1 & 1 & $1$ & $1$    & $2$ &   $3$&   $4$ & $5$ & $7$ & $9$ & $12$ & $15$ & $20$ & $25$ & $32$ & $40$ \\
\end{tabular}

\begin{center}
Table 1. Initial values of $(|\mathcal{A}_{n,\ell}|)_{n=1}^\infty$ for $0\leqslant \ell\leqslant 3$.
\end{center}

\begin{tabular}{c|ccccccccccccccccc}
$n$ & $1$ & $2$ & $3$ & $4$ & $5$ & $6$ & $7$ & $8$ & $9$ & $10$ & $11$ & $12$ & $13$ & $14$ & $15$ & $16$ & $17$   \\
\hline
$|\mathcal{A}^s_{n,0}|$ & 1 & 2 & 3 & 5        & $7$ &    $10$ &   $14$ & $19$ & $25$ & $33$ & $43$ & $55$ & $70$ & $88$ & $110$ & $137$ & $169$\\
$|\mathcal{A}^s_{n,1}|$ & 1 & 1 & $2$ & $3$    & $4$ &   $6$&   $8$ & $11$ & $14$ & $19$ & $24$ & $31$ & $39$ & $49$ & $61$ & $76$ & $93$\\
$|\mathcal{A}^s_{n,2}|$ & 1 & 1 & $1$ & $2$   & $3$ &   $4$&   $5$ & $7$ & $9$ & $12$ & $15$ & $19$ & $24$ & $30$ & $37$ & $46$ & $56$\\
$|\mathcal{A}^s_{n,3}|$  & 1 & 1 & $1$ & $1$    & $2$ &   $3$&   $4$ & $5$ & $6$ & $8$ & $10$ & $13$ & $16$ & $20$ & $24$ & $30$ & $36$\\
\end{tabular}

\begin{center}
Table 2. Initial values of $(|\mathcal{A}^s_{n,\ell}|)_{n=1}^\infty$ for $0\leqslant \ell\leqslant 3$.
\end{center}

\begin{tabular}{c|ccccccccccccccccc}
$n$ & $0$ & $1$ & $2$ & $3$ & $4$ & $5$ & $6$ & $7$ & $8$ & $9$ & $10$ & $11$ & $12$ & $13$ & $14$ & $15$ & $16$    \\
\hline
$|\mathcal{E}_{n,1}|$ & 1 & 1 & 2 & 3 & 5        & $7$ &    $11$ &   $15$ & $22$ & $30$ & $42$ & $56$ & $77$ & $101$ & $135$ & $176$ & $231$ \\
$|\mathcal{E}_{n,2}|$ & 1 & 0  & 1 & $1$ & $2$    & $2$ &   $4$&   $4$ & $7$ & $8$ & $12$ & $14$ & $21$ & $24$ & $34$ & $41$ & $55$ \\
$|\mathcal{E}_{n,3}|$ & 1 & 0 & 0 & $1$ & $1$   & $1$ &   $2$&   $2$ & $3$ & $4$ & $5$ & $6$ & $9$ & $10$ & $13$ & $17$ & $21$ \\
$|\mathcal{E}_{n,4}|$ & 1 & 0 & 0 & $0$ & $1$    & $1$ &   $1$&   $1$ & $2$ & $2$ & $3$ & $3$ & $5$ & $5$ & $7$ & $8$ & $11$ \\
\end{tabular}

\begin{center}
Table 3. Initial values of $(|\mathcal{E}_{n,\ell}|)_{n=0}^\infty$ for small $1\leqslant \ell\leqslant 4$.
\end{center}

\begin{tabular}{c|ccccccccccccccccc}
$n$ & $0$ & $1$ & $2$ & $3$ & $4$ & $5$ & $6$ & $7$ & $8$ & $9$ & $10$ & $11$ & $12$ & $13$ & $14$ & $15$ & $16$    \\
\hline
$|\mathcal{E}^d_{n,1}|$ & 1 & 1 & 1 & 2 & 2        & $3$ &    $4$ &   $5$ & $6$ & $8$ & $10$ & $12$ & $15$ & $18$ & $22$ & $27$ & $32$ \\
$|\mathcal{E}^d_{n,2}|$ & 1 & 0  & 1 & $1$ & $1$    & $2$ &   $2$&   $3$ & $3$ & $5$ & $5$ & $7$ & $8$ & $10$ & $12$ & $15$ & $17$ \\
$|\mathcal{E}^d_{n,3}|$ & 1 & 0 & 0 & $1$ & $1$   & $1$ &   $1$&   $2$ & $2$ & $3$ & $3$ & $4$ & $5$ & $6$ & $7$ & $9$ & $10$ \\
$|\mathcal{E}^d_{n,4}|$ & 1 & 0 & 0 & $0$ & $1$    & $1$ &   $1$&   $1$ & $1$ & $2$ & $2$ & $3$ & $3$ & $4$ & $4$ & $6$ & $6$ \\
\end{tabular}

\begin{center}
Table 4. Initial values of $(|\mathcal{E}^d_{n,\ell}|)_{n=0}^\infty$ for small $1\leqslant \ell\leqslant 4$.
\end{center}

\begin{tabular}{c|ccccccccccccccccc}
$n$ & $0$ & $1$ & $2$ & $3$ & $4$ & $5$ & $6$ & $7$ & $8$ & $9$ & $10$ & $11$ & $12$ & $13$ & $14$ & $15$ & $16$   \\
\hline
$|\mathcal{G}_{n,1}|$ & 1 & 1 & $2$ & $3$    & $5$ &   $7$ &   $11$ & $15$ & $22$ & $30$ & $42$ & $56$ & $77$ & $101$ & $135$ & $176$ & $231$\\
$|\mathcal{G}_{n,2}|$ & 1 & 1 & $1$ & $2$   & $3$ &   $4$&   $6$ & $8$ & $11$ & $15$ & $20$ & $26$ & $35$ & $45$ & $58$ & $75$ & $96$\\
$|\mathcal{G}_{n,3}|$ & 1 & 1 & $1$ & $1$    & $2$ &   $3$&   $4$ & $5$ & $7$ & $9$ & $12$ & $15$ & $20$ & $25$ & $32$ & $40$ & $51$\\
\end{tabular}

\begin{center}
Table 5. Initial values of $(|\mathcal{G}_{n,\ell}|)_{n=0}^\infty$ for small $1\leqslant \ell\leqslant 3$.
\end{center}

\begin{tabular}{c|ccccccccccccccccc}
$n$ & $0$ & $1$ & $2$ & $3$ & $4$ & $5$ & $6$ & $7$ & $8$ & $9$ & $10$ & $11$ & $12$ & $13$ & $14$ & $15$ & $16$   \\
\hline
$|\mathcal{H}_{n,1}|$ & 1 & 1 & $2$ & $3$    & $4$ &   $6$&   $8$ & $11$ & $14$ & $19$ & $24$ & $31$ & $39$ & $49$ & $61$ & $76$ & $93$\\
$|\mathcal{H}_{n,2}|$ & 1 & 1 & $1$ & $2$   & $3$ &   $4$&   $5$ & $7$ & $9$ & $12$ & $15$ & $19$ & $24$ & $30$ & $37$ & $46$ & $56$\\
$|\mathcal{H}_{n,3}|$ & 1 & 1 & $1$ & $1$    & $2$ &   $3$&   $4$ & $5$ & $6$ & $8$ & $10$ & $13$ & $16$ & $20$ & $24$ & $30$ & $36$\\
\end{tabular}

\begin{center}
Table 6. Initial values of $(|\mathcal{H}_{n,\ell}|)_{n=0}^\infty$ for small $1\leqslant \ell\leqslant 3$.
\end{center}


\ \\

\begin{thebibliography}{9999}
\bibitem{AG} S. A. Argyros and I. Gasparis, Unconditional structures of weakly null sequences, \textit{Trans. Amer. Math. Soc.} \textbf{353} (2001),  2019--2058.
\bibitem{BCF} K. Beanland, H. V. Chu, and C. E. Finch-Smith, Generalized Schreier sets, linear recurrence relation, and Turán graphs,  \textit{Fibonacci Quart.} \textbf{60} (2022), 352--356.
\bibitem{B} A. Bird, Schreier sets and the Fibonacci sequence, 
\url{https://outofthenormmaths.wordpress.com/
2012/05/13/jozef-schreier-schreier-sets-and-the-fibonacci-sequence/}.
\bibitem{CMX} H. V. Chu, S. J. Miller, and Z. Xiang, Higher order Fibonacci sequences from generalized Schreier sets, \textit{Fibonacci Quart.} \textbf{58} (2020), 249--253.
\bibitem{CIMSZ} H. V. Chu, N. Irmak, S. J. Miller, L. Szalay, and S. X. Zhang, Schreier multisets and the $s$-step Fibonacci sequences, to appear in Proceedings of the Integer Conference 2023. Available at: \url{https://arxiv.org/abs/2304.05409}.
\bibitem{C3} H. V. Chu, On a relation between Schreier-type sets and a modification of Turán graphs, \textit{Integers} \textbf{23} (2023), 1--10.
\bibitem{FN} V. Farmaki and S. Negrepontis, Schreier sets in Ramsey theory, \textit{Trans. Amer. Math. Soc.} \textbf{360} (2008), 849--880.
\bibitem{KKMW} M. Kolo\v{g}lu, G. S. Kopp, S. J. Miller, and Y. Wang, On the number of summands in Zeckendorf decompositions, \textit{Fibonacci Quart.} \textbf{49} (2011), 116--130. 
\bibitem{Od}
E. Odell, On schreier unconditional sequences, \textit{Contemp. Math.} \textbf{144} (1993),
197--201.
\bibitem{S} J. Schreier, Ein gegenbeispiel zur theorie der schwachen konvergentz, \textit{Studia Math.} \textbf{2}
(1962), 58--62.
\bibitem{Sl} N. J. A. Sloane et al., The On-Line Encyclopedia of Integer Sequences, 2023. Available at \url{https://oeis.org}.
\bibitem{W} Douglas B. West, \textit{Combinatorial Mathematics}, Cambridge University Press, 2021.
\end{thebibliography}
\end{document}